\documentclass[12pt]{article}

\usepackage{amssymb,amsmath,latexsym}
\usepackage{amsthm}
\usepackage{bm}
\usepackage{xcolor}
\usepackage{enumerate}
\usepackage[
  pdftex,
  pdfauthor = {Evan Donald and Jason Swanson},
  pdftitle = {
    An Aldous-Hoover type representation for row exchangeable arrays
  },
  hidelinks
]{hyperref}

\title{An Aldous-Hoover type representation for row exchangeable arrays}
\author{
  Evan Donald\\
  University of Central Florida\\
  \href{mailto:ev446807@ucf.edu}{ev446807@ucf.edu}
  \and
  Jason Swanson\\
  University of Central Florida\\
  \href{mailto:jason.swanson@ucf.edu}{jason.swanson@ucf.edu}
}
\date{}

\newtheorem{thm}{Theorem}[section]
\newtheorem{cor}[thm]{Corollary}
\newtheorem{prop}[thm]{Proposition}
\newtheorem{lemma}[thm]{Lemma}

\numberwithin{equation}{section}

\providecommand{\flr}[1]{\left\lfloor{#1}\right\rfloor}

\def\ol{\overline}

\def\wt{\widetilde}

\def\al{\alpha}
\def\be{\beta}
\def\ga{\gamma}
\def\de{\delta}
\def\ze{\zeta}
\def\la{\lambda}
\def\vpi{\varpi}
\def\si{\sigma}
\def\ph{\varphi}
\def\om{\omega}
\def\Om{\Omega}

\def\bfd{{\bf d}}

\def\bmmu{\bm{\mu}}
\def\bmpi{\bm{\pi}}
\def\bmX{{\bm X}}

\def\bN{\mathbb{N}}
\def\bR{\mathbb{R}}
\def\bS{\mathbb{S}}

\def\cB{\mathcal{B}}
\def\cC{\mathcal{C}}
\def\cF{\mathcal{F}}
\def\cG{\mathcal{G}}
\def\cL{\mathcal{L}}
\def\cM{\mathcal{M}}
\def\cR{\mathcal{R}}
\def\cS{\mathcal{S}}
\def\cT{\mathcal{T}}
\def\cU{\mathcal{U}}

\begin{document}

\maketitle

\begin{abstract}

In an array of random variables, each row can be regarded as a single,
sequence-valued random variable. In this way, the array is seen as a sequence of
sequences. Such an array is said to be row exchangeable if each row is an
exchangeable sequence, and the entire array, viewed as a sequence of sequences,
is exchangeable. We give a representation theorem, analogous to those of Aldous
and Hoover, which characterizes row exchangeable arrays. We then use this
representation theorem to address the problem of performing Bayesian inference
on row exchangeable arrays.

\bigskip



\end{abstract}

% \tableofcontents

\section{Introduction}

Consider a situation in which there are several agents, all from the same
population. Each agent undertakes a sequence of actions. These actions are
chosen according to the agent's particular tendencies. Although different agents
have different tendencies, there may be patterns in the population. We observe a
certain set of agents over a certain amount of time. Based on these
observations, we want to make probabilistic forecasts about the future behavior
of the agents.

We model this situation with an array of random variables. Let $S$ be a complete
and separable metric space and $\xi = \{\xi_{ij}: i, j \in \bN\}$ an array of
$S$-valued random variables. The space $S$ represents the set of possible
actions that the agents may undertake, and $\xi_{ij}$ represents the $j$th
action of the $i$th agent.

For a simple example, imagine a pressed penny machine, like those found in
museums or tourist attractions. For a fee, the machine presses a penny into a
commemorative souvenir. Now imagine the machine is broken, so that it mangles
all the pennies we feed it. Each pressed penny it creates is mangled in its own
way. Each has its own probability of landing on heads when flipped. The machine,
though, might have its own tendencies. For instance, it might tend to produce
pennies that are biased toward heads. In this situation, the agents are the
pennies and the actions are the heads and tails that they produce. We therefore
take $S = \{0, 1\}$ and $\xi_{ij} \in S$ would denote the result of the $j$th
flip of the $i$th penny created by the machine.

For an example with an infinite $S$, imagine the machine produces a
commemorative globe that fits in the palm of our hand. When we buy such a globe,
we might roll it like we would a die. When it comes to rest, whatever latitude
and longitude are at the top, we regard as the outcome of the roll. Ordinarily,
the machine produces globes that are uniformly weighted, so that every point on
the globe is equally likely to be the outcome of the roll. But suppose the
machine is broken, so that each globe is imbalanced and irregularly weighted in
its own unique way. In this case, we take $S = \bS^2$ to be the unit sphere in
$\bR^3$ and $\xi_{ij} \in S$ would denote the result of the $j$th roll of the
$i$th globe created by the machine.

In situations such as these, it would be natural to assume that $\xi$ has
certain symmetries. More specifically, we assume that
\begin{enumerate}[(i)]
  \item each row of the array, $\xi_i = \{\xi_{ij}: j \in \bN\}$, is an
        exchangeable sequence of $S$-valued random variables, and
  \item the sequence of rows, $\{\xi_i: i \in \bN\}$, is an exchangeable
        sequence of $S^\infty$-valued random variables.
\end{enumerate}
If $\xi$ satisfies (i) and (ii) above, then we say that $\xi$ is a 
\emph{row exchangeable} array. It is easy to see that $\xi$ is row exchangeable
if and only if $\{\xi_{\si(i),\tau_i(j)}\}$ and $\{\xi_{ij}\}$ have the same
finite-dimensional distributions whenever $\si$ and $\tau_i$ are (finite)
permutations of $\bN$.

In a collection of papers (see \cite{Aldous1981, Aldous1985, Hoover1979,
Hoover1982}), Aldous and Hoover considered different forms of exchangeability
for random arrays. One form they considered was \emph{separate exchangeability},
which means that $\{\xi_{\si(i), \tau(j)}\}$ and $\{\xi_{ij}\}$ have the same
finite-dimensional distributions whenever $\si$ and $\tau$ are permutations.
Clearly, row exchangeability implies separate exchangeability, but not
conversely.

Separately exchangeable arrays can be characterized by a certain representation
involving i.i.d.~uniform random variables. This representation theorem is given
below in Theorem \ref{T:AldousHoov}, and was originally proven independently by
Aldous \cite{Aldous1981,Aldous1985} and Hoover \cite{Hoover1979,Hoover1982}. A
proof can also be found in \cite[Corollary 7.23]{Kallenberg2005}. The version
given here appears in \cite{Kallenberg1989}.

\begin{thm}\label{T:AldousHoov}
  Let $\xi = \{\xi_{ij}: i, j \in \bN\}$ be an array of $S$-valued random
  variables. Then $\xi$ is separately exchangeable if and only if there exists a
  measurable function $g: \bR^4 \to S$ and an i.i.d.~collection of random
  variables $\{\al, \be_i, \eta_j, \la_ {ij}: i, j \in \bN\}$, uniformly
  distributed on $(0, 1)$, such that the array $\{g(\al, \be_i, \eta_j,
  \la_{ij})\}$ has the same finite-dimensional distributions as $\xi$.
\end{thm}

The main results of the present work are twofold. First, we establish an
analogous representation theorem for row exchangeable arrays (see Theorem
\ref{T:rowExRepr}). We then use this representation theorem to address the
problem of performing Bayesian inference on row exchangeable arrays. For the
latter, we start by proving a de Finetti theorem for row exchangeability (see
Theorem \ref{T:ex=>IDP}). According to this theorem, if $\xi$ is row
exchangeable, then there exists a sequence of random measure $\mu_1, \mu_2,
\ldots$ on $S$ and a random measure $\vpi$ on $M_1 = M_1(S)$, the space of
probability measures on $S$, such that
\begin{itemize}
  \item given $\vpi$, the sequence $\mu_1, \mu_2, \ldots$ is i.i.d.~with
        distribution $\vpi$, and
  \item for each $i$, given $\mu_i$, the sequence $\xi_{i1}, \xi_{i2}, \ldots$
        is i.i.d.~with distribution $\mu_i$.
\end{itemize}
We call the random measures $\mu_i$ the \emph{row distributions of $\xi$}, and
we call $\vpi$ the \emph{row distribution generator}. To keep our formulas
concise, we adopt the following notation:
\[
  X_{in} = \begin{pmatrix}
    \xi_{i1} & \cdots & \xi_{in}
  \end{pmatrix}, \qquad
  \bmX_{mn} = \begin{pmatrix}
    X_{1n} \\ \vdots \\ X_{mn}
  \end{pmatrix}
  = \begin{pmatrix}
    \xi_{11} & \cdots & \xi_{1n} \\
    \vdots & \ddots & \vdots \\
    \xi_{m1} & \cdots & \xi_{mm}
  \end{pmatrix}, \qquad
  \bmmu_m = \begin{pmatrix}
    \mu_1 \\ \vdots \\ \mu_m
  \end{pmatrix}.
\]
Our objective is to make inferences about the future of the process $\xi$ based
on past observations. That is, if $M, N, N' \in \bN$ and $N < N'$, then we wish
to compute
\begin{equation}\label{main-goal}
  \cL(\xi_{ij}: i \le M, \; N < j \le N' \mid \bmX_{MN})
\end{equation}
in terms of the prior distribution of $\xi$. Here, the notation $\cL(X \mid Y)$
denotes the regular conditional distribution of $X$ given $Y$. We also adopt the
convention that a variable with a $0$ subscript is omitted. For example, when $m
= 1$, we understand $\cL(\mu_m \mid \bmX_{MN}, \bmmu_{m - 1})$ to mean
$\cL(\mu_1 \mid \bmX_{MN})$.

Our main result concerning this problem is the following.

\begin{thm}\label{T:main}
  The conditional distribution \eqref{main-goal} is entirely determined by the
  conditional distributions,
  \begin{equation}\label{main-goal3}
    \cL(\mu_m \mid X_{mN}, \ldots, X_{MN}, \bmmu_{m - 1}),
  \end{equation}
  where $1 \le m \le M$.
\end{thm}

The proof of Theorem \ref{T:main} (see Section \ref{S:comp-post}) is
constructive and provides us with a way of using \eqref{main-goal3} to compute
\eqref{main-goal}.

In follow-up work (see \cite{Donald2025a}), we take on the question of how to
assign a non-informative prior to $\xi$. According to our de Finetti theorem,
this is equivalent to assigning a prior to $\vpi$. Note that $\vpi$ is a random
measure on $M_1(S)$. That is, $\vpi$ is an $M_1(M_1(S))$-valued random variable.
Its prior is therefore an element of $M_1(M_1(M_1(S)))$. Our approach to
assigning a prior will be to use a variant of Ferguson's \cite{Ferguson1973}
Dirichlet process.

\section{Background}

\subsection{Random measures and kernels}

If $(X, \cU)$ is a topological space, then $\cB(X) = \si(\cU)$ denotes the Borel
$\si$-algebra on $X$. We may sometimes denote $\cB(X)$ by $\cB_X$. We use
$\cR^d$ to denote the Borel $\si$-algebra on $\bR^d$ with the Euclidean
topology. If $(S, \cS)$ is a measurable space and $B \in \cS$, then $\cS|_B =
\{A \cap B: A \in \cS\} = \{A \in \cS: A \subseteq B\}$. We let $\bR^* =
[-\infty, \infty]$ denote the extended real line, equipped with the topology
generated by the metric $\rho(x,y) = |\tan^{-1}(x) - \tan^{-1} (y)|$. We use
$\cR^*$ to denote $\cB(\bR^*)$. Note that $\cR^* = \{A \subseteq \bR^*: A \cap
\bR \in \cR\}$.

Let $S$ be a complete and separable metric space with metric $\bfd$ and let $\cS
= \cB(S)$. Let $M = M(S)$ denotes the space of $\si$-finite measures on $S$. For
$B \in \cS$, we define the projection $\pi_B: M \to \bR^*$ by $\pi_B(\nu) = \nu
(B)$. If $T$ is a set and $\mu: T \to M$, then we adopt the notation $\mu(t, B)
= (\mu(t))(B)$. Let $\cM = \cM(S)$ be the $\si$-algebra on $M$ defined by $\cM =
\si(\{\pi_B: B \in \cS\})$. Note that $\cM$ is the collection of all sets of the
form $\{\nu \in M: \nu(B) \in A\}$, where $B \in \cS$ and $A \in \cB([0,
\infty])$. Given a probability space, $(\Om, \cF, P)$, a \emph{random measure on
$S$} is a function $\mu: \Om \to M$ that is $(\cF, \cM)$-measurable.

Let $M_1 = M_1(S)$ be the set of all probability measures on $S$. Note that $M_1
= \{\nu \in M: \nu(S) \in \{1\}\}$. Hence, $M_1 \in \cM$ and we may define
$\cM_1$ by $\cM_1 = \cM|_ {M_1}$. A \emph{random probability measure on $S$} is
a function $\mu: \Om \to M_1$ that is $(\cF, \cM_1)$-measurable, or
equivalently, a random measure taking values in $M_1$.

We equip $M_1$ with the Prohorov metric, $\bmpi$, which metrizes weak
convergence. Since $S$ is complete and separable, $M_1$ is complete and
separable under $\bmpi$. It can be shown that $\cM_1 = \cB(M_1)$. (See, for
example, \cite[Theorem 2.3]{Gaudard1989}.)

Unless otherwise specified, we will equip $S^\infty$ with the metric
$\bfd^\infty$ defined by
\[
  \bfd^\infty(x, y) = \sum_{n = 1}^\infty \frac{\bfd(x_n, y_n) \wedge 1}{2^n}.
\]
Note that $x \to y$ in $S^\infty$ if and only if $x_n \to y_n$ in $S$ for all
$n$. In particular, the metric $\bfd^\infty$ induces the product topology on
$S^\infty$. Consequently, $\cS^\infty = \cB(S^\infty)$ and $S^\infty$ is
separable. In fact, $(S^\infty, \bfd^\infty)$ is a complete and separable metric
space.

Let $(T, \cT)$ be a measurable space. A function $\mu: T \to M$ is called a 
\textit{kernel from $T$ to $S$} if $\mu(\cdot, B)$ is $(\cT, \cR^*)$-measurable
for all $B \in \cS$. By \cite[Lemma 1.37]{Kallenberg1997}, given a function
$\mu: T \to M$ and a $\pi$-system $\cC$ such that $\cS = \si(\cC)$, the
following are equivalent:
\begin{enumerate}[(i)]
  \item $\mu$ is a kernel.
  \item $\mu(\cdot, B)$ is $(\cT, \cR^*)$-measurable for all $B \in \cC$.
  \item $\mu$ is $(\cT, \cM)$-measurable.
\end{enumerate}
A \textit{probability kernel} is a kernel taking values in $M_1$. Note that a
random measure on $S$ is just a kernel from $\Om$ to $S$, and a random
probability measure on $S$ is just a probability kernel from $\Om$ to $S$.

It can be shown that the function $\mu \mapsto \mu^\infty$ mapping $M_1$ to $M_1
(S^\infty)$ is measurable. Therefore, if $\mu$ is a random probability measure
on $S$, then $\mu^\infty$ is a random probability measure on $S^\infty$. In
fact, the random variable $\mu^\infty$ is $\si(\mu)$-measurable.

\subsection{Regular conditional distributions}

If $X$ is an $S$-valued random variable and $\mu$ is a probability measure on
$S$, then we write $X \sim \mu$, as usual, to mean that $X$ has distribution
$\mu$. We also use $\cL(X)$ to denote the distribution of $X$, so that $X \sim
\mu$ and $\cL(X) = \mu$ are synonymous. If $\cL(X) = \mu$ and $B \in \cS$, then
we use $\cL(X; B)$ to denote $\mu(B)$.

Let $\cG \subseteq \cF$ be a $\si$-algebra. A \emph{regular conditional
distribution for $X$ given $\cG$} is a random probability measure $\mu$ on $S$
such that $P(X \in A \mid \cG) = \mu(A)$ a.s.~for all $A \in \cS$. Since $S$ is
a complete and separable metric space, regular conditional distributions exist.
In fact, such a $\mu$ exists whenever $X$ takes values in a standard Borel
space. Moreover, $\mu$ is unique in the sense that if $\mu$ and $\wt \mu$ are
two such random measures, then $\mu = \wt \mu$ a.s. That is, with probability
one, $\mu(B) = \wt \mu(B)$ for all $B \in \cS$. The random measure $\mu$ may be
chosen so that it is $\cG$-measurable.

We use the notation $\cL(X \mid \cG)$ to refer to the regular conditional
distribution of $X$ given $\cG$. If $\mu$ is a random probability measure, the
notation ${X \mid \cG} \sim \mu$ means $\cL(X \mid \cG) = \mu$ a.s. We use the
semicolon notation in the same way as for unconditional distributions so that,
for instance, $\cL({X \mid \cG}; A) = P(X \in A \mid \cG)$ a.s.

Let $(T, \cT)$ be a measurable space and $Y$ a $T$-valued random variable. Then
there exists a probability kernel $\mu$ from $T$ to $S$ such that $X \mid \si
(Y) \sim \mu(Y)$. Moreover, $\mu$ is unique in the sense that if $\mu$ and $\wt
\mu$ are two such probability kernels, then $\mu = \wt \mu$, $\cL(Y)$-a.e. In
this case, we will typically omit the $\si$, and write, for instance, that $\cL
(X \mid Y) = \mu(Y)$. Note that for fixed $y \in T$, the probability measure
$\mu(y, \cdot)$ is not uniquely determined, since the kernel $\mu$ is only
unique $\cL(Y)$-a.e. Nonetheless, if a particular $\mu$ has been fixed, we will
use the notation $\cL(X \mid Y = y)$ to denote the probability measure $\mu(y,
\cdot)$.

\subsection{Conditional independence}

A sequence $\xi = \{\xi_i: i \in \bN\}$ of $S$-valued random variables is
\textit{conditionally i.i.d.}~if there exists a $\si$-algebra $\cG \subseteq
\cF$ and a random probability measure $\nu$ on $S$ such that
\begin{equation}\label{condiidDef}
  P(\xi \in A \mid \cG) = \nu^\infty(A) \text{ a.s.},
\end{equation}
for all $A \in \cS^\infty$.

If there is a random probability measure $\mu$ such that $P(\xi \in A \mid \mu)
= \mu^\infty(A)$, then $\xi$ is conditionally i.i.d.~given $\cG = \si(\mu)$. The
converse is given by the following lemma.

\begin{lemma}\label{L:G=sig(mu)}
  If $\xi$ satisfies \eqref{condiidDef}, then there exists a random
  probability measure $\mu$ on $S$ such that
  \begin{equation}\label{condiid}
    P(\xi \in A \mid \mu) = \mu^\infty(A) \text{ a.s.},
  \end{equation}
  for all $A \in \cS^\infty$. Moreover, $\mu$ can be chosen so that $\mu \in
  \cG$ and $\mu = \nu$ a.s.
\end{lemma}

\begin{proof}
  Suppose there exists a $\si$-algebra $\cG \subseteq \cF$ and a random
  probability measure $\nu$ on $S$ such that $P(\xi \in A \mid \cG) =
  \nu^\infty(A)$ a.s., for all $A \in \cS^\infty$. Let $(\Om, \ol \cG, \ol P)$
  be the completion of $(\Om, \cG, P)$, so that $\nu^\infty(A)$ is $(\ol \cG,
  \cR)$-measurable for all $A \in \cS^\infty$. This means that $\nu^\infty$ is a
  kernel from $(\Om, \ol \cG)$ to $S^\infty$, and that $\nu^\infty$ is $\ol
  \cG$-measurable. For any $B \in \cS$,
  \[
    \nu(B) = \nu^\infty(B \times S^\infty)
  \]
  is $(\ol \cG, \bR)$-measurable. Hence, $\nu$ is $\ol \cG$-measurable. Choose
  $\mu: \Om \to M_1(S)$ such that $\mu$ is $\cG$-measurable and $\mu = \nu$ a.s.
  Thus, $\mu^\infty = \nu^\infty$ a.s. Since $\mu^\infty \in \si(\mu)$, we have
  \[
    P(\xi \in A \mid \mu)
      = E[P(\xi \in A \mid \cG) \mid \mu]
      = E[\mu^\infty(A) \mid \mu]
      = \mu^\infty(A) \text{ a.s.},
  \]
  for all $A \in \cS^\infty$.
\end{proof}

\begin{cor}\label{C:G=sig(mu)}
  If $\xi$ is conditionally i.i.d.~given $\cG$, then
  \[
    P\left(
      \bigcap_{i = 1}^n \{\xi_i \in A_i\}
    \;\middle|\; \cG \right)
      = \prod_{i = 1}^n P(\xi_i \in A_i \mid \cG) \quad \text{a.s.},
  \]
  for all $n \in \bN$ and all $A_i \in \cS$.
\end{cor}

\begin{proof}
  By Lemma \ref{L:G=sig(mu)},
  \[
    P\left(
      \bigcap_{i = 1}^n \{\xi_i \in A_i\}
    \;\middle|\; \cG \right)
      = \mu^\infty(A_1 \times \cdots \times A_n)
      = \prod_{i = 1}^n \mu(A_i) \quad \text{a.s.}
  \]
  If $A = S^{i - 1} \times A_i \times S^\infty$, then $\mu(A_i) = \mu^\infty(A)
  = P(\xi \in A \mid \cG) = P(\xi_i \in A_i \mid \cG)$ a.s.
\end{proof}

\subsection{Exchangeability and empirical measures}\label{S:emp-meas}

A sequence $\xi = \{\xi_i: i \in \bN\}$ of $S$-valued random variables is
\textit{exchangeable} if
\[
  (\xi_{k_1}, \ldots, \xi_{k_m}) \overset{d}{=} (\xi_1, \ldots, \xi_m)
\]
whenever $k_1, \ldots, k_m$ are distinct elements of $\bN$. A function $\si: \bN
\to \bN$ is called a \textit{(finite) permuatation} if $\si$ is bijective and
there exists $n_0 \in \bN$ such that $\si(n) = n$ for all $n \ge n_0$. The
sequence $\xi$ is exchangeable if and only if $\{\xi_{\si(i)}\}$ and $\{\xi_i\}$
have the same finite-dimensional distributions whenever $\si$ is a permutation.

By \cite[Theorem 1.1]{Kallenberg2005}, we have that $\xi$ is conditionally
i.i.d.~if and only if $\xi$ is exchangeable. In that case, the random measure
$\mu$ in \eqref{condiid} is a.s.~unique, $\si(\xi)$-measurable, and satisfies
$\mu_n(B) \to \mu(B)$ a.s., for all $B \in \cS$, where
\[
  \mu_n = \frac 1 n \sum_{i = 1}^n \de_{\xi_i}
\]
are the empirical measures (see \cite[Proposition 1.4]{Kallenberg2005}). In
Proposition \ref{P:empToPrior} below, we show that $\mu_n \to \mu$ a.s.~in $M_1$
under the Prohorov metric.

\begin{lemma}\label{L:cesaroMble}
  Let $T$ a convex subset of a real vector space. Let $\rho$ be a metric on $T$
  such that $(T, \rho)$ is complete and separable, and for all $n \in \bN$, the
  function $y \mapsto n^{-1} \sum_{i = 1}^n y_i$ is continuous from $T^n$ to
  $T$. Let $g: S \to T$ be a continuous function, and fix $y_0 \in T$. Define
  $\ph_n: S^\infty \to T$ by
  \[
    \ph_n(x) = n^{-1} \sum_{i = 1}^n g(x_i).
  \]
  Define $B = \{x \in S^\infty: \lim_{n \to \infty} \ph_n(x) \text{ exists}\}$,
  and define $\la: S^\infty \to T$ by
  \[
    \la(x) = \begin{cases}
      \lim_{n \to \infty} \ph_n(x) &\text{if $x \in B$},\\
      y_0 &\text{if $x \notin B$}.
    \end{cases}
  \]
  Then $\ph_n$ is continuous, $B \in \cS^\infty$, and $\la$ is $(\cS^\infty,
  \cB(T))$-measurable.
\end{lemma}

\begin{proof}
  Define $\Phi_n: T^n \to T$ by $\Phi(y) = n^{-1} \sum_{i = 1}^n y_i$, so that
  $\Phi_n$ is continuous. Let $\pi_n: S^\infty \to S^n$ be the projection
  mapping $x \mapsto (x_1, \ldots, x_n)$, so that $\pi_n$ is continuous. Define
  $g_n: S^n \to T^n$ by $g_n(x) = (g(x_1), \ldots ,g(x_n))$. Note that $\ph_n =
  \Phi_n \circ g_n \circ \pi_n$, which shows that $\ph_n$ is continuous.

  Since $(T, \rho)$ is complete and separable, and each $\ph_n$ is measurable,
  it follows that $B$ is measurable. Thus, for every $n$, the function $x
  \mapsto \ph_n(x) 1_B + y_0 1_{B^c}$ is $(\cS^\infty, \cB(T))$-measurable.
  Since $(T, \rho)$ is separable and $\la$ is the pointwise limit of these
  functions, it follows that $\la$ is $(\cS^\infty, \cB(T))$-measurable.
\end{proof}

\begin{prop}\label{P:empToPrior}
  Suppose $\xi$ is exchangeable. Let $\mu_n$ be the empirical measures and $\mu$
  the a.s.~unique measure in \eqref{condiid}. Then, with probability one, $\mu_n
  \to \mu$ weakly.
\end{prop}

\begin{proof}
  We apply Lemma \ref{L:cesaroMble} with $T = M_1$, $\rho = \bmpi$
  the Prohorov metric, $g(x) = \de_x$, and $y_0$ arbitrary. In this case, by
  \cite[Theorem 11.4.1]{Dudley1989}, given any $m \in M_1$, we have $m^\infty(B)
  = 1$ and
  \[
    m^{\infty}(\{x \in S^\infty: \la(x) = m\}) = m^{\infty} (\la^{-1}(m)) = 1.
  \]
  In other words, the empirical distributions of an i.i.d.~sequence with
  distribution $m$ converge weakly to $m$ a.s.

  Note that $\mu_n = \ph_n(\xi)$ and
  \[
    P(\xi \in B) = E[P(\xi \in B \mid \mu)] = E[\mu^\infty(B)] = 1,
  \]
  since $\mu^\infty(B) \equiv 1$. Hence, $\mu_n \to \la(\xi)$ a.s., so it
  suffices to prove that $\la(\xi) = \mu$ a.s.

  For this, let $G = \{(x, m) \in S^\infty \times M_1: \la(x) = m\}$ be the
  graph of $\la$. Since $\cM_1 = \cB(M_1)$, we have that $\la$ is $(\cS^\infty,
  \cM_1)$-measurable, which implies $G \in \cS^\infty \otimes \cM_1$. Thus, by
  \cite[Theorem 5.4]{Kallenberg1997},
  \begin{multline*}
    P(\la(\xi) = \mu) = E[P((\xi, \mu) \in G \mid \mu)]
    = E\left[{
      \int_S 1_G(x, \mu)\,\mu^{\infty}(\cdot, dx)
    }\right]\\
    = E\left[{
      \int_S 1_{\la^{-1}(\mu)}(x)\,\mu^{\infty}(\cdot, dx)
    }\right]
    = E[\mu^{\infty}(\la^{-1}(\mu))]
    = 1,
  \end{multline*}
  since $\mu^{\infty}(\la^{-1}(\mu)) \equiv 1$.
\end{proof}

\subsection{Conditional law of large numbers}

We can use the preceding results to prove the following conditional version of
the law of large numbers.

\begin{prop}\label{P:cond-LLN}
  Let $(T, \cT)$ be a measurable space. Suppose $Y$ is a $T$-valued random
  variable, $\xi$ is conditionally i.i.d.~given $\cG$, and $f: T \times S \to
  \bR$ is a measurable function such that $E|f(Y, \xi_1)| < \infty$. If $Y \in
  \cG$ and $Z$ is any version of $E[f(Y, \xi_1) \mid \cG]$ then
  \begin{equation}\label{cond-LLN-1}
    P\left(
      \lim_{n \to \infty} \frac 1 n \sum_{i = 1}^n f(Y, \xi_i) = Z
    \;\middle|\;
      \cG
    \right) = 1 \quad\text{a.s.}
  \end{equation}
  In particular,
  \begin{equation}\label{cond-LLN-2}
    \frac 1 n \sum_{i = 1}^n f(Y, \xi_i) \to E[
      f(Y, \xi_1) \mid \cG
    ] \quad \text{a.s.}
  \end{equation}
\end{prop}

\begin{proof}
  Taking expectations in \eqref{cond-LLN-1} yields \eqref{cond-LLN-2}. It thus
  suffices to prove \eqref{cond-LLN-1}. We first prove \eqref{cond-LLN-1} under
  the assumption that $S=\bR$ with the Euclidean metric and $f(t, x)=x$. We
  begin by applying Lemma \ref{L:cesaroMble} with $T = \bR$, $\rho$ the
  Euclidean metric, $g(x) = x$, and $y_0 = 0$, obtaining $\ph_n$, $B$, and
  $\la$. Suppose $m \in M_1(\bR)$ and $\int_\bR |x| \, m(dx) < \infty$. Let $\ol
  x = \int_\bR x \, m(dx)$. By the law of large numbers, $m^\infty(B) = 1$ and
  \[
    m^{\infty}(\{x \in \bR^\infty: \la(x) = \ol x\})
      = m^{\infty}(\la^{-1}(\ol x)) = 1.
  \]
  Since $\xi$ is conditionally i.i.d., there exists a random probability measure
  $\mu$ such that $\xi \mid \cG \sim \mu^\infty$. Thus,
  \[
    E \int_\bR |x| \, \mu(dx)
      = E \int_{\bR^\infty} |\pi_1(x)| \,\mu^\infty(dx)
      = E[ E[|\xi_1| \mid \cG] ]
      = E|\xi_1| < \infty.
  \]
  Hence, there exists $\Om^* \in \cF$ such that $P(\Om^*) = 1$ and
  \[
    \int_\bR |x| \, \mu(\om, dx) < \infty,
  \]
  for all $\om \in \Om^*$. For any such $\om$, define $\ol x_\om = \int_\bR x \,
  \mu(\om, dx)$. As above, we then have $\mu^\infty(\om, B) = 1$ and $\mu^\infty
  (\om, \la^{-1}(\ol x_\om)) = 1$. Since $Z = E[\xi_1 \mid \cG] = \int_\bR x \,
  \mu(dx)$ a.s., we may assume that $Z(\om) = \ol x_\om$ for all $\om \in
  \Om^*$.

  Now, note that $n^{-1} \sum_{i = 1}^n \xi_i = \ph_n(\xi) \to Z$ if and only if
  $\xi \in B$ and $\la(\xi) = Z$. Thus,
  \[
    P\left(
      \lim_{n \to \infty} \frac 1 n \sum_{i = 1}^n \xi_i = Z \;\middle|\; \cG
    \right)
      = E[1_B(\xi) 1_G(\xi, Z) \mid \cG],
  \]
  where $G = \{(x, y): \la(x) = y\}$ is the graph of $\la$. By \cite[Theorem
  5.4]{Kallenberg1997}, since $Z \in \cG$, this is equal to $\int_\bR 1_B(x) 1_G
  (x, Z) \, \mu^\infty(dx)$. Thus, for any $\om \in \Om^*$, we have
  \begin{align*}
    P\left(
      \lim_{n \to \infty} \frac 1 n \sum_{i = 1}^n \xi_i = Z \;\middle|\; \cG
    \right)(\om)
      &= \int_\bR 1_B(x) 1_G(x, \ol x_\om) \, \mu^\infty(\om, dx)\\
    &= \int_B 1_{\la^{-1}(\ol x_\om)}(x) \, \mu^\infty(\om, dx)\\
    &= \mu^\infty(\om, B \cap \la^{-1}(\ol x_\om)) = 1,
  \end{align*}
  and this proves \eqref{cond-LLN-1} under the assumption that $S = \bR$ with
  the Euclidean metric and $f(t, x) = x$.

  To prove the general result, let $\xi_i' = f(Y, \xi_i)$. By \cite[Theorem 5.4]
  {Kallenberg1997},
  \[
    P\left(\bigcap_{j = 1}^n \{\xi_i' \in B_i\} \;\middle|\; \cG \right)
    = \int_{S^n} \prod_{j = 1}^n 1_{B_i}(f(Y, x_i)) \, \nu^n(dx)
    = \prod_{j = 1}^n \int_S 1_{B_i}(f(Y, x_i)) \, \nu(dx_i).
  \]
  Hence, $\{\xi_i'\}$ are conditionally i.i.d.~given $\cG$. The result therefore
  follows by applying the first part of the proof to $\{\xi_i'\}$.
\end{proof}

\section{Bayesian inference for row exchangeable arrays}

In this section, we prove our main results, which include an Aldous-Hoover type
representation (Theorem \ref{T:rowExRepr}) and the proof of Theorem \ref{T:main}
concerning Bayesian inference.

\subsection{An Aldous-Hoover type representation theorem}

Theorem \eqref{T:rowExRepr} below is the analogue of Theorem \ref{T:AldousHoov}
for row exchangeable arrays. It characterizes row exchangeable arrays, and
provides a representation for them in terms of i.i.d.~uniform random variables.

\begin{thm}\label{T:rowExRepr}
  The array $\xi$ is row exchangeable if and only if there exists a measurable
  function $f:\bR^3\to S$ and an i.i.d.~collection of random variables $\{\al,
  \be_i, \la_{ij}: i, j \in \bN\}$, uniformly distributed on $(0, 1)$, such that
  the array $\{f(\al, \be_i, \la_{ij})\}$ has the same finite-dimensional
  distributions as $\xi$.
\end{thm}

\begin{proof}
  The if direction is trivial. For the only if direction, suppose $\xi$ is row
  exchangeable. Note that row exchangeability implies separate exchangeability,
  so that by Theorem \ref{T:AldousHoov}, there exists a measurable function
  $g: \bR^4 \to S$ and an i.i.d.~collection of random variables $\{\al, \be_i,
  \eta_j, \la_{ij}: i, j \in \bN\}$, uniformly distributed on $(0, 1)$, such
  that the array $\{g(\al, \be_i, \eta_j, \la_{ij})\}$ has the same
  finite-dimensional distributions as $\xi$.

  For $n \in \bN$, define $d_n: (0, 1) \to \{0, 1\}$ by $d_n(x) = \flr{2^n x} -
  2 \flr{2^{n - 1} x}$, so that $d_n(x)$ is the $n$-th digit of the canonical
  binary expansion of $x$. Define $\ph: (0, 1) \to (0, 1)^2$ by
  \[
    \ph(x) = \left({
      \sum_{n = 1}^\infty 2^{-n} d_{2n - 1}(x),
      \sum_{n = 1}^\infty 2^{-n} d_{2n}(x)
    }\right).
  \]
  Note that $\ph$ is measurable, and that whenever $U$ is uniform on $(0, 1)$,
  it follows that $\ph_1(U)$ and $\ph_2(U)$ are independent and also uniform on
  $(0, 1)$.

  Define $f: \bR^3 \to S$ by $f(a, b, z) = g(a, b, \ph_1(z), \ph_2(z))$, and let
  $\ga_{ij} = \ph_1(\la_{ij})$ and $\ze_{ij} = \ph_2(\la_{ij})$, so that $f(\al,
  \be_i, \la_{ij}) = g(\al, \be_i, \ga_{ij}, \ze_{ij})$. It now suffices to show
  that $\{g(\al, \be_i, \ga_{ij}, \ze_{ij})\}$ and $\{g(\al, \be_i, \eta_j,
  \la_{ij})\}$ have the same finite-dimensional distributions.

  Fix $m, n \in \bN$ and let $A_{ij} \in \cS$ for $i \le m$ and $j \le n$. For
  each $i \in \bN$, choose a permutation $\tau_i$ that maps $\{1, \ldots ,n\}$
  to $\{in + 1, \ldots, in + n\}$. Since $\xi$ is row exchangeable, we have
  \begin{align*}
    P\left({
      \bigcap_{i = 1}^m \bigcap_{j = 1}^n
        \{g(\al, \be_i, \eta_j, \la_{ij}) \in A_{ij}\}
    }\right)
      &= P\left({
        \bigcap_{i = 1}^m \bigcap_{j = 1}^n \{\xi_{ij} \in A_{ij}\}
      }\right)\\
    &= P\left({
        \bigcap_{i = 1}^m \bigcap_{j = 1}^n \{\xi_{i, \tau_i(j)} \in A_{ij}\}
      }\right)\\
    &= P\left({
      \bigcap_{i = 1}^m \bigcap_{j = 1}^n
        \{g(\al, \be_i, \eta_{\tau_i(j)}, \la_{i,\tau_i(j)}) \in A_{ij}\}
    }\right)
  \end{align*}
  Since the mapping $(i, j) \mapsto \tau_i(j)$ from $\{1, \ldots, m\} \times 
  \{1, \ldots, n\}$ to $\bN$ is injective, it follows that $\{\al, \be_i, \eta_
  {\tau_i (j)}, \la_{i, \tau_i(j)}: i \le m, j \le n\}$ is an i.i.d.~collection
  of random variables. In particular, $\{\al, \be_i, \eta_{\tau_i(j)}, \la_{i,
  \tau_i(j)}: i \le m, j \le n\}$ and $\{\al, \be_i, \ga_{ij}, \ze_{ij}: i \le
  m, j \le n\}$ have the same distribution. Thus,
  \[
    P\left({
      \bigcap_{i = 1}^m\bigcap_{j = 1}^n
        \{g(\al, \be_i, \eta_{\tau_i(j)}, \la_{i, \tau_i(j)}) \in A_{ij}\}
    }\right)
      = P\left({
      \bigcap_{i = 1}^m \bigcap_{j = 1}^n
        \{g(\al, \be_i, \ga_{ij}, \ze_{ij}) \in A_{ij}\}
    }\right).
  \]
  Combined with the above, this shows that $\{g(\al, \be_i, \ga_{ij}, \ze_
  {ij})\}$ and $\{g(\al, \be_i, \eta_j, \la_{ij})\}$ have the same
  finite-dimensional distributions.
\end{proof}

\subsection{A de Finetti theorem for row exchangeability}

Theorem \ref{T:ex=>IDP} below is a de Finetti type theorem for row exchangeable
arrays. It shows that the law of a row exchangeable array, $\xi$, is governed by
its row distributions, $\mu_i$, and its row distribution generator, $\vpi$.

\begin{thm}\label{T:ex=>IDP}
  Suppose $\xi$ is a row exchangeable, infinite array of $S$-valued random
  variables. Then there exists an a.s.~unique random probability measure $\vpi$
  on $M_1$ and an a.s.~unique sequence $\mu=\{\mu_i: i \in \bN\}$ of random
  probability measures on $S$ such that ${\mu \mid \vpi} \sim \vpi^\infty$ and
  ${\xi_i \mid \mu_i} \sim \mu_i^\infty$ for all $i \in \bN$.
\end{thm}

\begin{proof}
  Fix $i \in \bN$. Since $\xi_i$ is exchangeable, Proposition \ref{P:empToPrior}
  implies that there exists an a.s.~unique random probability measure $\mu_i$ on
  $S$ such that, with probability one, $n^ {-1}\sum_{j = 1}^n \de_{\xi_{ij}} \to
  \mu_i$ weakly as $n \to \infty$, and ${\xi_i \mid \mu_i} \sim \mu_i^\infty$.

  According to the proof of Proposition \ref{P:empToPrior}, we may take $\mu_i =
  \la(\xi_i)$, where $\la: S^\infty \to M_1$ is measurable and does not depend
  on $i$. Let $\mu = \{\mu_i\}_{i = 1}^\infty$. Since $\mu_i = \la (\xi_i)$ and
  the sequence $\{\xi_i\}_{i = 1}^\infty$ is exchangeable, it follows that
  $\{\mu_i\}_{i = 1}^\infty$ is an exchangeable sequence of $M_1$-valued random
  variables.

  It now follows from Proposition \ref{P:empToPrior} that there exists
  an a.s.~unique random probability measure $\vpi$ on $M_1$ such that, with
  probability one, $n^{-1}\sum_{i = 1}^n \de_{\mu_i} \to \vpi$ weakly as $n \to
  \infty$, and ${\mu \mid \vpi} \sim \vpi^\infty$.
\end{proof}

The relation ${\xi_i \mid \mu_i} \sim \mu_i^\infty$ in Theorem \ref{T:ex=>IDP}
concerns only one fixed row at a time. Theorem \ref{T:condOnBmu} below
generalizes this to multiple rows. It states that $\vpi$ and $\bmX_{mn}$ are
independent given $\bmmu_m$, and gives an explicit expression for the law of
$\bmX_{mn}$ given $\bmmu_m$.

\begin{thm}\label{T:condOnBmu}
  If $\xi$ is row exchangeable, then
  \[
    \cL(\bmX_{mn} \mid \vpi, \bmmu_m) = \cL(\bmX_{mn} \mid \bmmu_m)
      = \prod_{i = 1}^m \mu_i^n \quad \text{a.s.,}
  \]
  for all $m, n \in \bN$.
\end{thm}

\begin{proof}
  Fix $m, n \in \bN$. For $i \le m$ and $j \le n$, let $A_{ij} \in \cS$.
  According to the proof of Proposition \ref{P:empToPrior}, we can write $\vpi$
  as a measurable function of $\mu$, so that $\vpi \in \si(\mu)$. Thus,
  \[
    \prod_{i = 1}^m \mu_i^n \in \si(\bmmu_m)
      \subseteq \si(\vpi, \bmmu_m)
      \subseteq \si(\vpi, \mu)
      = \si(\mu).
  \]
  It therefore suffices to show that
  \begin{equation}\label{condOnBmu}
    P\left(
      \bigcap_{i = 1}^m \bigcap_{j = 1}^n \{\xi_{ij} \in A_{ij}\}
        \;\middle|\; \mu
    \right) = \prod_{i = 1}^m \prod_{j = 1}^n \mu_i(A_{ij}) \quad \text{a.s.}
  \end{equation}
  By Theorem \ref{T:rowExRepr}, we may assume that $\xi_{ij} = f(\al,
  \be_i, \la_{ij})$, where $f: \bR^3 \to S$ is measurable and $\{\al, \be_i,
  \la_{ij}: i, j \in \bN\}$ is an i.i.d.~collection of random variables, uniform
  on $(0, 1)$.

  Fix $i \le m$ and $B \in \cS$. As noted in Section \ref{S:emp-meas},
  \cite[Proposition 1.4]{Kallenberg2005} implies that $n^{-1} \sum_{j = 1}^n
  \de_{\xi_{ij}}(B) \to \mu_i(B)$ a.s. On the other hand, $\de_{\xi_{ij}}(B) =
  h_B(\al, \be_i, \la_{ij})$, where $h_B = 1_B \circ f$. Since $h_B$ is bounded
  and $\{h_B(\al, \be_i, \la_{ij})\}_{j = 1}^\infty$ is conditionally
  i.i.d.~given $\al$ and $\be_i$, Proposition \ref{P:cond-LLN} implies
  \[
    \frac 1 n \sum_{j = 1}^n \de_{\xi_{ij}}(B)
      = \frac 1 n \sum_{j = 1}^n h_B(\al, \be_i, \la_{ij})
      \to E[h_B(\al, \be_i, \la_{i1}) \mid \al, \be_i]
      = \int_0^1 h_B(\al, \be_i, x) \, dx \quad \text{a.s.}
  \]
  Thus,
  \begin{equation}\label{mu_i-repr}
    \mu_i(B) = \int_0^1 h_B(\al, \be_i, x) \, dx \quad \text{a.s.,}
  \end{equation}
  which implies that the random variable $\mu_i(B)$ is measurable with respect
  to $\ol{\si(\al, \be_i)}$, the completion of $\si(\al, \be_i)$ with respect to
  $P$. Since this is true for every $B\in\cS$, it follows that $\mu_i \in 
  \ol{\si(\al, \be_i)}$. Thus, we may choose $\nu_i \in \si(\al, \be_i)$ such
  that $\mu_i = \nu_i$ a.s. Letting $\nu = \{\nu_i\}_{i = 1}^\infty$ and $\be = 
  \{\be_i\}_{i = 1}^\infty$, and noting that $\nu \in \si(\al, \be)$ and $\mu =
  \nu$ a.s., we have
  \begin{multline*}
    P\left(
      \bigcap_{i = 1}^m \bigcap_{j = 1}^n \{\xi_{ij} \in A_{ij}\}
        \;\middle|\; \mu
    \right)
      = P\left(
          \bigcap_{i = 1}^m \bigcap_{j = 1}^n \{\xi_{ij} \in A_{ij}\}
            \;\middle|\; \nu
        \right)\\
    = E\left[
        P\left(
          \bigcap_{i = 1}^m \bigcap_{j = 1}^n \{\xi_{ij} \in A_{ij}\}
            \;\middle|\; \al, \be
        \right)
          \;\middle|\; \nu
      \right]
    = E\left[
        P\left(
          \bigcap_{i=1}^m\bigcap_{j=1}^n \{\xi_{ij} \in A_{ij}\}
            \;\middle|\; \al, \be
        \right)
          \;\middle|\; \mu
      \right] \quad \text{a.s.}
  \end{multline*}
  Since $\xi_{ij} = f(\al, \be_i, \la_{ij})$, it follows that $\{\xi_{ij}\}$ is
  conditionally i.i.d.~given $\al, \be$. By Corollary \ref{C:G=sig(mu)},
  \[
    P\left(
      \bigcap_{i = 1}^m \bigcap_{j = 1}^n \{\xi_{ij} \in A_{ij}\}
        \;\middle|\; \al, \be
    \right) = \prod_{i = 1}^m\prod_{j = 1}^n
      P(\xi_{ij} \in A_{ij} \mid \al, \be) \quad \text{a.s.}
  \]
  By \eqref{mu_i-repr},
  \[
    P(\xi_{ij} \in A_{ij} \mid \al, \be)
      = E[h_{A_{ij}}(\al, \be_i, \la_{ij}) \mid \al, \be]
      = \int_0^1 h_{A_{ij}}(\al, \be_i, x) \, dx
      = \mu_i(A_{ij}) \quad \text{a.s.}
  \]
  Combining the last three displays gives
  \[
    P\left(
      \bigcap_{i = 1}^m \bigcap_{j = 1}^n \{\xi_{ij} \in A_{ij}\}
        \;\middle|\; \mu
    \right)
      = E\left[
          \prod_{i = 1}^m \prod_{j = 1}^n \mu_i(A_{ij})
            \;\middle|\; \mu
        \right]
      = \prod_{i = 1}^m \prod_{j = 1}^n \mu_i(A_{ij}) \quad \text{a.s.},
  \]
  establishing \eqref{condOnBmu} and finishing the proof.
\end{proof}

\subsection{Computing the posterior distribution}\label{S:comp-post}

In this section, we prove Theorem \ref{T:main}. The proof relies on Theorems
\ref{T:muSuffices} and \ref{T:MarkovProp} below. The former states that the
conditional law of the array $\{\xi_{ij}: i, j \in \bN\}$ is entirely determined
by the conditional law of the row distributions $\{\mu_i: i \in \bN\}$. The
latter gives a certain conditional independence property. Namely, the
conditional law of the rows above a given $m$ depends only on the row
distributions $\mu_1, \ldots, \mu_m$, and not on the observed data $\{\xi_{ij}:
i \le m, j \in \bN\}$.

\begin{thm}\label{T:muSuffices}
  Let $\xi$ be row exchangeable. Fix $M \in \bN$. For each $i \in \{1, \ldots,
  M\}$, let $N_i, O_i \in \bN$ with $N_i < O_i$, and let $A_{ij} \in \cS$ for
  $i \le M$ and $N_i < j \le O_i$. Let $\cG$ be a sub-$\si$-algebra of $\si
  (X_{1 N_1}, X_{2 N_2}, \ldots, X_{M N_M})$. Then
  \begin{equation}\label{muSuffices}
    P\left(
      \bigcap_{i = 1}^M \bigcap_{j = N_i + 1}^{O_i} \{\xi_{ij} \in A_{ij}\}
        \; \middle| \; \cG
    \right)
      = E\left[
        \prod_{i = 1}^M \prod_{j = N_i + 1}^{O_i} \mu_i(A_{ij})
          \; \middle| \; \cG
      \right].
  \end{equation}
\end{thm}

\begin{proof}
  Let $A \in \cG$ have the form $A = \bigcap_{i = 1}^M \bigcap_{j = 1}^{N_i}
  \{\xi_{ij} \in A_{ij}\}$. Let $Z$ denote the left-hand side of
  \eqref{muSuffices}. Then
  \begin{multline*}
    E[Z 1_A] = E\left[ P\left(
          A \cap \bigcap_{i = 1}^M \bigcap_{j = N_i + 1}^{O_i} {
            \{\xi_{ij} \in A_{ij}\}
          } \; \middle| \; \cG
        \right) \right]
      = P\left(
          \bigcap_{i = 1}^M \bigcap_{j = 1}^{O_i} \{\xi_{ij} \in A_{ij}\}
        \right)\\
    = E\left[ P\left(
          \bigcap_{i = 1}^M \bigcap_{j = 1}^{O_i} \{\xi_{ij} \in A_{ij}\}
            \; \middle| \; \bmmu_M
        \right) \right]
      = E\left[ \prod_{i = 1}^M \prod_{j = 1}^{O_i} \mu_i(A_{ij}) \right],
  \end{multline*}
  by Theorem \ref{T:condOnBmu}. On the other hand,
  \begin{multline*}
    E\left[ \left(
          \prod_{i = 1}^M \prod_{j = N_i + 1}^{O_i} \mu_i(A_{ij})
        \right) 1_A \right]
      = E\left[ E\left[
          \left(
            \prod_{i = 1}^M \prod_{j = N_i + 1}^{O_i} \mu_i(A_{ij})
          \right) 1_A
            \; \middle| \; \bmmu_M
        \right] \right]\\
    = E\left[ \left(
          \prod_{i = 1}^M \prod_{j = N_i + 1}^{O_i} \mu_i(A_{ij})
        \right) P(A \mid \bmmu_M) \right]
      = E\left[ \prod_{i = 1}^M \prod_{j = 1}^{O_i} \mu_i(A_{ij}) \right].
  \end{multline*}
  Hence, the two displays are equal. By the $\pi$-$\la$ theorem, they are equal
  for all $A \in \cG$, and this proves the claim.
\end{proof}

\begin{thm}\label{T:MarkovProp}
  Let $\xi$ be row exchangeable. Fix $m, M, N \in \bN$ with $m < M$. Then
  $\bmX_{mN}$ and $\{ (\mu_j, X_{jN})\}_{j = m + 1}^M$ are conditionally
  independent given $\bmmu_m$.
\end{thm}

\begin{proof}
  Let $A_{ij} \in \cS$ for $i \le M$ and $j \le N$, and let $B_i \in \cM_1$ for
  $m < i \le M$. We adopt the shorthand $A_i = A_{i1} \times \cdots \times A_
  {iN}$. By Theorem \ref{T:condOnBmu},
  \begin{multline*}
    P(
      \mu_{m + 1} \in B_{m + 1}, \ldots, \mu_M \in B_M,
      X_{1N} \in A_1, \ldots, X_{MN} \in A_M
      \mid \bmmu_M
    )\\
    = \prod_{i = m + 1}^M 1_{B_i}(\mu_i) \prod_{i = 1}^M \mu_i^N(A_i).
  \end{multline*}
  Hence,
  \begin{multline*}
    P(
      \mu_{m + 1} \in B_{m + 1}, \ldots, \mu_M \in B_M,
      X_{1N} \in A_1, \ldots, X_{MN} \in A_M
    \mid \bmmu_m)\\
    = E\bigg[
      \prod_{i = m + 1}^M 1_{B_i}(\mu_i) \prod_{i = 1}^M \mu_i^N(A_i)
    \;\bigg|\; \bmmu_m \bigg]
    = \bigg(\prod_{i = 1}^m \mu_i^N(A_i)\bigg) E\bigg[
      \prod_{i = m + 1}^M 1_{B_i}(\mu_i) \mu_i^N(A_i)
    \;\bigg|\; \bmmu_m \bigg].
  \end{multline*}
  Taking $A_{ij} = S$ for all $i \le m$ gives
  \begin{multline*}
    P(
      \mu_{m + 1} \in B_{m + 1}, \ldots, \mu_M \in B_M,
      X_{m + 1, N} \in A_{m + 1}, \ldots, X_{MN} \in A_M
    \mid \bmmu_m)\\
    = E\bigg[
      \prod_{i = m + 1}^M 1_{B_i}(\mu_i) \mu_i^N(A_i)
    \;\bigg|\; \bmmu_m \bigg].
  \end{multline*}
  Since $P(X_{1N} \in A_1, \ldots, X_{mN} \in A_m \mid \bmmu_m) = \prod_{i = 1}^m
  \mu_i^N(A_i)$, this completes the proof.
\end{proof}

\begin{proof}[Proof of Theorem \ref{T:main}]
  According to Theorem \ref{T:muSuffices} with $N_i = N$, $O_i = N'$, and $\cG =
  \si(\bmX_{MN})$, we can compute the posterior distribution \eqref{main-goal},
  provided we can compute $\cL(\bmmu_M \mid \bmX_{MN})$. Note that
  \[
    \cL(\bmmu_M \mid \bmX_{MN}; d\nu)
      = \cL(
        \mu_2, \ldots, \mu_M \mid \bmX_{MN}, \mu_1 = \nu_1;
        d\nu_2 \cdots d\nu_M
      ) \, \cL(\mu_1 \mid \bmX_{MN}; d\nu_1).
  \]
  Iterating this, we see that we can determine $\cL(\bmmu_M \mid \bmX_{MN})$ if
  we know
  \begin{equation}\label{main-goal2}
    \cL(\mu_m \mid \bmX_{MN}, \bmmu_{m - 1}), \quad 1 \le m \le M.
  \end{equation}
  Theorem \ref{T:MarkovProp} shows that in \eqref{main-goal2}, the first $m - 1$
  rows of $\bmX_{MN}$ can be omitted. Hence, the conditional distribution
  \eqref{main-goal} is entirely determined by the distributions given in
  \eqref{main-goal3}.
\end{proof}

\bibliographystyle{plain}
\bibliography{idp-paper}
\addcontentsline{toc}{section}{References}

\end{document}